\documentclass[11pt]{amsart}
\usepackage{amssymb, amsthm, amsmath, gensymb, dsfont, cite}
\usepackage{graphicx, comment}
\usepackage[small]{caption}
\usepackage{subcaption}
\usepackage{epsfig}
\usepackage{tikz, pgfplots, float}
\usepackage{setspace}
\usepackage{amsfonts}

\usepackage[utf8]{inputenc}
\usepackage{fullpage}
\usepackage{framed}

\usepackage{enumerate}
\usepackage{url}
\usepackage[breaklinks]{hyperref}
\usepackage{cleveref}
\hypersetup{
	colorlinks = true, 
	urlcolor = cyan, 
	linkcolor = teal, 
	citecolor = cyan 
}

\newtheorem{theorem}{Theorem}

\newtheorem{lemma}[theorem]{Lemma}
\newtheorem{claim}[theorem]{Claim}
\crefname{claim}{claim}{claims}

\newtheorem{corollary}[theorem]{Corollary}

\title{A note on hypergraphs with asymmetric Ramsey properties}

\author{Vladimir Sviridenkov}

\address{Karlsruhe Institute of Technology, Englerstraße 2, D-76131 Karlsruhe, Germany}

\email{sviridenkov.vova@gmail.com}

\begin{document}
\maketitle

\begin{abstract}
Let $r,\ell\geq2$ be integers. Given $r$-graphs $G$ and $F_1,\dots,F_\ell$, we write $G\to(F_1,\dots,F_\ell)$ if every $\ell$-edge-coloring of $G$ yields a monochromatic copy of $F_i$ in the $i$th color for some $1\leq i\leq\ell$, otherwise we write $G\not\to(F_1,\dots,F_\ell)$. The Ramsey number $R(F_1,\dots,F_\ell)$ is the minimum number of vertices in an $r$-graph $G$ satisfying $G\to(F_1,\dots,F_\ell)$.

In this note we prove that for any integers $t_1\geq\dots\geq t_\ell>r$, there exists an $r$-graph $G$ such that $G\not\to(K^{(r)}_{t_1},\dots,K^{(r)}_{t_\ell})$ but $G\to(K^{(r)}_s,K^{(r)}_{t_\ell-1})$, where $s=R(K^{(r)}_{t_1},\dots,K^{(r)}_{t_\ell})-1$. This extends recent work by Mendon\c{c}a, Miralaei, and Mota, who established the statement for $r=2$.
\end{abstract}

\section{Introduction}
Let $r\geq2$ be an integer. An $r$-uniform hypergraph (or $r$-graph) $F$ is an ordered pair $(V(F),E(F))$, where $V(F)$ is the vertex set and $E(F)\subseteq\binom{V(F)}{r}$ is the edge set. Let $K^{(r)}_t$ be a complete $r$-graph on $t$ vertices that is, an $r$-graph in which every $r$-subset of vertices forms an edge. When $r=2$, an $r$-graph is simply referred to as a graph with the superscript omitted. Throughout this note, all hypergraphs are assumed to be finite.

Given $r$-graphs $G$ and $F_1,\dots,F_\ell$, we say that $G$ is Ramsey for $(F_1,\dots,F_\ell)$, written as $G\to(F_i)_\ell$, if every $\ell$-edge-coloring of $G$ yields a monochromatic copy of $F_i$ in the $i$-th color for some $1\leq i\leq\ell$. Otherwise, we write $G\not\to(F_i)_\ell$. The Ramsey number $R(F_1,\dots,F_\ell)$ is defined to be the minimum number of vertices in an $r$-graph $G$ satisfying $G\to(F_i)_\ell$. For general background on graph and hypergraph Ramsey theory, see, e.g.,~\cite{morris2026,mubayi2020}.

Classical Ramsey theory studies how large a complete hypergraph has to be to satisfy a prescribed
Ramsey property. A complementary structural question asks whether a hypergraph can simultaneously satisfy some Ramsey properties while failing
others. In particular, can a hypergraph fail to be Ramsey for a pair $(F_1,F_2)$, but still
be Ramsey for another pair $(F_1',F_2')$, where $F_1'$ contains many copies of $F_1$ and $F_2'$ is obtained by deleting a vertex from $F_2$? Mendon\c{c}a,
Miralaei, and Mota~\cite{mendonca2025} recently answered this question affirmatively
for graph cliques. More precisely, they proved the following.

\begin{theorem}[Mendon\c{c}a--Miralaei--Mota~\cite{mendonca2025}]
\label{previous}
For any integer $t\geq2$, there exists a graph $G$ such that
\[G\not\to(K_t,K_t) \quad\text{and}\quad G\to(K_s,K_{t-1}),\]
where $s=R(K_t,K_t)-1$.
\end{theorem}

A natural first attempt to prove~\Cref{previous} is to consider random graphs. However, the threshold for a random graph to fail to be Ramsey for $(K_t,K_t)$~\cite{rodl1995,nenadov2016} is much lower than the threshold for it to be Ramsey for $(K_s,K_{t-1})$~\cite{christoph2025}. Namely, if a random graph is not Ramsey for $(K_t,K_t)$ then almost surely it is also not Ramsey for $(K_s,K_{t-1})$. Mendon\c{c}a, Miralaei, and Mota~\cite{mendonca2025} advanced beyond this barrier by considering the primal graph of a random hypergraph. This approach yields a significantly lower threshold above which the resulting graph exhibits the desired asymmetric Ramsey properties.

For an $r$-graph $F$, let $v(F)$ and $e(F)$ denote the numbers of vertices and edges of $F$, respectively. The maximum $r$-density of $F$ is defined as
\[m_r(F):=\begin{cases}
0 & \text{if $e(F)=0$,}\\
1/r & \text{if $e(F)=1$,}\\
\max_{J\subseteq F,v(J)>r}\frac{e(J)-1}{v(J)-r} & \text{if $e(F)\geq2$.}
\end{cases}\]

In this note, we establish a hypergraph extension of their result, stated as follows.

\begin{theorem}
\label{asymmetric_Ramsey}
Let $r,\ell\geq2$ be integers. For any integers $t_1\geq\dots\geq t_\ell>r$ and any $r$-graph $F$ satisfying $m_r(F)<m_r(K^{(r)}_{t_\ell})$, there exists an $r$-graph $G$ such that
\[G\not\to(K^{(r)}_{t_1},\dots,K^{(r)}_{t_\ell}) \quad\text{and}\quad G\to(K^{(r)}_{s},F),\]
where $s=R(K^{(r)}_{t_1},\dots,K^{(r)}_{t_\ell})-1$.
\end{theorem}

It is easy to verify that $m_r(K^{(r)}_t)=\frac{\binom{t}{r}-1}{t-r}$ for $t>r$, which yields that $m_r(K^{(r)}_t)>m_r(K^{(r)}_{t'})$ for all $t>t'\geq r$. The following is therefore a straightforward corollary of~\Cref{asymmetric_Ramsey}.

\begin{corollary}
Let $r,\ell\geq2$ be integers. For any integers $t_1\geq\dots\geq t_\ell>t\geq r$ and $s<R(K^{(r)}_{t_1},\dots,K^{(r)}_{t_\ell})$, there exists an $r$-graph $G$ such that $G\not\to(K^{(r)}_{t_1},\dots,K^{(r)}_{t_\ell})$ but $G\to(K^{(r)}_{s},K^{(r)}_t)$.
\end{corollary}

Our proof of~\Cref{asymmetric_Ramsey} builds on the idea of Mendon\c{c}a, Miralaei, and Mota~\cite{mendonca2025}, but adapts it to the setting of the primal $r$-graph of a random hypergraph of higher uniformity. Although the statement of~\Cref{asymmetric_Ramsey} resembles its graph counterpart, extending the argument to hypergraphs is not immediate and requires additional technical notions, which we introduce below.

Let $s\geq t\geq r\geq2$ be integers and let $H$ be an $s$-graph. The primal $r$-graph of $H$, denoted by $G[H,r]$, is defined as the $r$-graph on the vertex set $V(H)$ and the edge set $\{B\in\binom{A}{r}:\,A\in E(H)\}$. We say that $H$ is $r$-linear if every $r$-subset of vertices is contained in at most one edge. Note that when $r=2$ this is the usual definition of a linear hypergraph. We say that $H$ is $(r,t)$-conformal if, for every copy of $K^{(r)}_t$ in $G[H,r]$, there exists an edge $A\in E(H)$ that contains all its vertices. A crucial observation is that if an $s$-graph is $r$-linear and $(r,t)$-conformal, then its primal $r$-graph is not Ramsey for a certain tuple of complete $r$-graphs. This will be elaborated in the next section, where we prove~\Cref{asymmetric_Ramsey}.

For integers $n\geq s\geq2$ and some real $0\leq p\leq1$, let $H(n,s,p)$ be a random $s$-graph on the vertex set $[n]$, where each $s$-subset of $[n]$ appears as an edge independently with probability $p$.~\Cref{asymmetric_Ramsey} is derived from the following theorem on dense $s$-subgraphs of $H(n,s,p)$ and their primal $r$-graphs.

\begin{theorem}
\label{random}
Let $s\geq t>r\geq2$ be fixed integers. Let $H=H(n,s,p)$. Then the following holds.
\begin{enumerate}[(i)]
\item\label{(i)} For any $r$-graph $F$ with $e(F)\geq2$, there is some constant $\delta>0$ such that, if
\[p\gg n^{r-s-1/m_r(F)},\] then with high probability
\[G[H_0,r]\to(K^{(r)}_s,F)\]
holds for every $H_0\subseteq H$ with $e(H_0)\geq(1-\delta)e(H)$.

\item\label{(ii)} If $n^{-s}\ll p\ll n^{r-s-1/m_r(K^{(r)}_t)}$, then with high probability there is an $r$-linear $(r,t)$-conformal $H_0\subseteq H$ with $e(H_0)=(1-o(1))e(H)$.
\end{enumerate}
\end{theorem}

The rest of the note is organized as follows. In~\Cref{main_theorem} we derive~\Cref{asymmetric_Ramsey} using~\Cref{random}. In~\Cref{first_statement} and~\Cref{second_statement}, we prove statements~\eqref{(i)} and~\eqref{(ii)} of~\Cref{random}, respectively. Some minor concluding remarks are stated in~\Cref{concluding}.

\section{Proof of~\Cref{asymmetric_Ramsey}}
\label{main_theorem}
\begin{proof}[Proof of~\Cref{asymmetric_Ramsey}]
Let $r,\ell\geq2$ and $t_1\geq\dots\geq t_\ell>r$ be integers. Let $s=R(K^{(r)}_{t_1},\dots,K^{(r)}_{t_\ell})-1$. Note that $s\geq t_\ell$. Let $F$ be an $r$-graph satisfying $m_r(F)<m_r(K^{(r)}_{t_\ell})$.

Assume first that $e(F)<2$. Let $G$ be obtained from a copy of $K^{(r)}_s$ by adding isolated vertices so that $v(G)\geq v(F)$. For any red-blue edge-coloring of $G$, either all edges are red, in which case we have a red copy of $K^{(r)}_s$, or there is at least one blue edge, in which case we have a blue copy of $F$. On the other hand, by the choice of $s$ there exists an $\ell$-edge-coloring of $K^{(r)}_s$, and hence of $G$, without monochromatic copy of $K^{(r)}_{t_i}$ in the $i$th color for all $1\leq i\leq\ell$. Thus, $G\to(K^{(r)}_s,F)$ but $G\not\to(K^{(r)}_{t_1},\dots,K^{(r)}_{t_\ell})$.

Now assume that $e(F)\geq2$. Let $n^{-s}\leq n^{r-s-1/m_r(F)}\ll p\ll n^{r-s-1/m_r(K^{(r)}_{t_\ell})}$ and $H=H(n,s,p)$. By~\Cref{random} there exists some constant $\delta>0$, such that with high probability
\begin{enumerate}
\item $G[H_0,r]\to(K^{(r)}_s,F)$ holds for every $H_0\subseteq H$ with $e(H_0)\geq(1-\delta)e(H)$;

\item there is an $r$-linear $(r,t_\ell)$-conformal $H_0\subseteq H$ with $e(H_0)=(1-o(1))e(H)$.
\end{enumerate}
Therefore, there is an $r$-linear $(r,t_\ell)$-conformal $s$-graph $H_0$ with $G[H_0,r]\to(K^{(r)}_s,F)$. It remains to show that $G[H_0,r]\not\to(K^{(r)}_{t_1},\dots,K^{(r)}_{t_\ell})$. Since $s=R(K^{(r)}_{t_1},\dots,K^{(r)}_{t_\ell})-1$, for every copy of $K^{(r)}_s$ in $G[H_0,r]$ corresponding to an edge of $H_0$, we can fix an $\ell$-edge-coloring without monochromatic copy of $K^{(r)}_{t_i}$ in the $i$th color for all $1\leq i\leq\ell$. Note that this yields an $\ell$-edge-coloring of $G[H_0,r]$. Indeed, by the definition of primal $r$-graphs and the fact that $H_0$ is $r$-linear, every edge of $G[H_0,r]$ lies in exactly one edge of $H_0$, and thus is colored exactly once. Moreover, fix any $1\leq i\leq\ell$ and consider a copy $K$ of $K^{(r)}_{t_i}$ in $G[H_0,r]$. Since $t_i\geq t_\ell>r$, $V(K)$ can be covered by a sequence of $t_\ell$-subsets of $V(K)$ such that every two consecutive ones intersect in at least $r$ vertices. As $H_0$ is $(r,t_\ell)$-conformal, every $t_\ell$-subset of $V(K)$ is contained in an edge of $H_0$. From the $r$-linearity of $H_0$ it follows that $V(K)$ is contained in a single edge of $H_0$, i.e., $K$ is contained in a copy of $K^{(r)}_s$ corresponding to an edge of $H_0$. Then by the construction of the coloring, $K$ is not monochromatic in the $i$th color. Accordingly, $G[H_0,r]\not\to(K^{(r)}_{t_1},\dots,K^{(r)}_{t_\ell})$.
\end{proof}

\section{Proof of~\Cref{random}: the first statement}
\label{first_statement}
In this section we prove the statement~\eqref{(i)} of~\Cref{random}. We begin with two preliminary lemmas. The first is a well-known fact in the literature, see, e.g.,~\cite[Theorem~11]{gugelmann2017}.

\begin{lemma}
\label{mono_copies}
For any $r$-graphs $T$ and $F$ there is a constant $\delta=\delta(T,F)>0$, such that every red-blue edge-coloring of $K^{(r)}_n$ with $n\geq R(T,F)$ contains at least $\delta n^{v(T)}$ red copies of $T$ or at least $\delta n^{v(F)}$ blue copies of $F$.
\end{lemma}

The second is a hypergraph container lemma. The method of hypergraph containers was introduced independently by Balogh, Morris, and Samotij~\cite{balogh2015}, and by Saxton and Thomason~\cite{saxton2015}. Here we use the version from~\cite[Theorem~2.3]{saxton2015}.

\begin{lemma}
\label{container}
Let $F$ be an $r$-graph with $e(F)\geq2$ and let $\delta>0$. Then for some $\alpha=\alpha(F,\delta)>0$ and every $n\geq\alpha$, there exists a collection $\mathcal{S}$ of $r$-graphs on the vertex set $[n]$ called sources and a collection $\mathcal{C}=\{C(S):\,S\in\mathcal{S}\}$ of $r$-graphs on the vertex set $[n]$ called containers, such that
\begin{enumerate}[(a)]
\item\label{(a)} for every $S\in\mathcal{S}$, $e(S)\leq\alpha n^{r-1/m_r(F)}$;

\item\label{(b)} for every $C\in\mathcal{C}$, the number of copies of $F$ in $C$ is smaller than $\delta n^{v(F)}$;

\item\label{(c)} for every $F$-free $r$-graph $G$ on the vertex set $[n]$, there exists $S\in\mathcal{S}$, such that $S\subseteq G\subseteq C(S)$.
\end{enumerate}
\end{lemma}

\subsection{Proof of~\Cref{random} statement~\eqref{(i)}}
Let $s>r\geq2$ be fixed integers and let $F$ be an $r$-graph with $e(F)\geq2$. Suppose that $n$ is sufficiently large. Let $\delta>0$ be the constant given by~\Cref{mono_copies}, such that every red-blue coloring of $K^{(r)}_n$ contains at least $\delta n^{s}$ red copies of $K^{(r)}_s$ or at least $\delta n^{v(F)}$ blue copies of $F$. Applying~\Cref{container} to $F$ and $\delta$, we obtain a constant $\alpha>0$, a collection $\mathcal{S}$ of sources, and a collection $\mathcal{C}$ of containers satisfying the properties~\eqref{(a)},~\eqref{(b)}, and~\eqref{(c)}.

Observe that for every $C\in\mathcal{C}$, its complement $\overline{C}$ contains at least $\delta n^s$ copies of $K^{(r)}_s$. Indeed, consider a red-blue edge-coloring of $K^{(r)}_n$ in which the edges of $\overline{C}$ are colored red and the edges of $C$ are colored blue. By item~\eqref{(b)} of~\Cref{container} we know that there are fewer than $\delta n^{v(F)}$ blue copies of $F$. Then by~\Cref{mono_copies} there must be at least $\delta n^s$ copies of $K^{(r)}_s$ in $\overline{C}$.

Let $H=H(n,s,p)$, where $p\gg n^{r-s-1/m_r(F)}$. Given an $r$-graph $G$, we write $G\sqsubset H$ if there are $e(G)$ distinct edges $A_1,\dots,A_{e(G)}\in E(H)$ such that, letting $E(G)=\{E_1,\dots,E_{e(G)}\}$, $E_i\subseteq A_i$ for all $1\leq i\leq e(G)$. Moreover, for every $C\in\mathcal{C}$, we let
\[X_{\overline{C}}=\left\{A\in\binom{[n]}{s}:\,\text{$A\in E(H)$ and $A$ spans a copy of $K^{(r)}_s$ in $\overline{C}$}\right\}.\]

\begin{claim}
\label{change_events}
For any $H_0\subseteq H$ with $e(H_0)\geq(1-\delta)e(H)$, if $G[H_0,r]\not\to(K^{(r)}_s,F)$, then there exist $S\in\mathcal{S}$ and $C=C(S)\in\mathcal{C}$ with $S\subseteq C$, such that $S\sqsubset H$ and $|X_{\overline{C}}|\leq\delta e(H)$.
\end{claim}
\begin{proof}[Proof of~\Cref{change_events}]
Let $H_0\subseteq H$ be a spanning $s$-subgraph with $e(H_0)\geq(1-\delta)e(H)$, and suppose that $G[H_0,r]\not\to(K^{(r)}_s,F)$. Then, there exists a red-blue edge coloring of $G[H_0,r]$ which contains neither a red copy of $K^{(r)}_s$ nor a blue copy of $F$. Since every $A\in E(H_0)$ corresponds to a copy of $K^{(r)}_s$ in $G[H_0,r]$, there must be a $B_A\in\binom{A}{r}$ that is a blue edge in $G[H_0,r]$. For every $A\in E(H_0)$ we choose such a $B_A$ and let $G'\subseteq G[H_0,r]$ be the spanning $r$-subgraph consisting of the chosen blue edges. Note that we have $G'\sqsubset H$. Since $G'$ is an $F$-free $r$-graph on the vertex set $[n]$, there exist some $S\in\mathcal{S}$ and $C=C(S)\in\mathcal{C}$ such that $S\subseteq G'\subseteq C$. It follows immediately that $S\sqsubseteq H$. Moreover, observe that $X_{\overline{C}}\cap E(H_0)=\emptyset$. Indeed, for every $A\in E(H_0)$ there is a $B_A\in\binom{A}{r}$ in $G'$ and thus not in $\overline{C}$, meaning that $A$ does not span a copy of $K^{(r)}_s$ in $\overline{C}$. Now since $X_{\overline{C}}\cap E(H_0)=\emptyset$ and $e(H_0)\geq(1-\delta)e(H)$, it follows that $|X_{\overline{C}}|\leq e(H)-e(H_0)\leq\delta e(H)$.
\end{proof}

Let $\mathbb{P}_\text{bad}$ be the probability that there is an $H_0\subseteq H$ with $e(H_0)\geq(1-\delta)e(H)$ and $G[H_0,r]\not\to(K^{(r)}_s,F)$. It remains to show that $\mathbb{P}_\text{bad}$ tends to $0$ as $n$ tends to infinity. By~\Cref{change_events} and the union bound, we have that
\begin{align*}
\mathbb{P}_\text{bad}&\leq\mathbb{P}\left(\exists\,S\in\mathcal{S}\text{ and }C=C(S)\in\mathcal{C}\text{, such that }S\sqsubset H\text{ and }|X_{\overline{C}}|\leq\delta e(H)\right)\\
&\leq\sum_{S\in\mathcal{S}}\mathbb{P}\left(S\sqsubset H\text{ and }|X_{\overline{C(S)}}|\leq\delta e(H)\right)\\
&\leq\mathbb{P}\left(e(H)\geq n^sp/2\right)+\sum_{S\in\mathcal{S}}\mathbb{P}\left(S\sqsubset H\text{ and }|X_{\overline{C(S)}}|\leq\delta n^sp/2\right).
\end{align*}

Since $p\gg n^{r-s-1/m_r(F)}\geq n^{-s}$ and $s\geq3$, we have that $1\ll\mathbb{E}(e(H))=\binom{n}{s}p\leq n^sp/6$. Applying Chernoff's bound (see, e.g.,~\cite{mitzenmacher2017}),
\[\mathbb{P}\left(e(H)\geq n^sp/2\right)\leq\mathbb{P}\left(e(H)\geq(1+1/3)\mathbb{E}(e(H))\right)=o(1).\]
Observe further that, for any $S\in\mathcal{S}$, the events $S\sqsubset H$ and $|X_{\overline{C(S)}}|\leq\delta n^sp/2$ are independent. Indeed, the events $S\sqsubset H$ and $|X_{\overline{C(S)}}|\leq\delta n^sp/2$ depend only on the members in
\[\mathcal{A}_1=\left\{A\in\binom{[n]}{s}:\,\binom{A}{r}\cap E(S)\neq\emptyset\right\}\quad\text{and}\quad\mathcal{A}_2=\left\{A\in\binom{[n]}{s}:\,\binom{A}{r}\subseteq E(\overline{C(S)})\right\},\]
respectively. Since $E(S)$ and $E(\overline{C(S)})$ are disjoint, we have that $\mathcal{A}_1\cap\mathcal{A}_2=\emptyset$. Accordingly,
\begin{equation}
\label{P_bad}
\mathbb{P}_\text{bad}\leq o(1)+\sum_{S\in\mathcal{S}}\mathbb{P}\left(S\sqsubset H\right)\cdot\mathbb{P}\left(|X_{\overline{C(S)}}|\leq\delta n^sp/2\right).
\end{equation}

Next we bound $\mathbb{P}\left(S\sqsubset H\right)$ and $\mathbb{P}\left(|X_{\overline{C(S)}}|\leq\delta n^sp/2\right)$ separately. First, say $E(S)=\{B_1,\dots,B_{e(S)}\}$ and let $Y$ denote the number of tuples $(A_1,\dots,A_{e(S)})$ of $e(S)$ distinct edges in $H$ such that $B_i\subseteq A_i$ for all $1\leq i\leq e(S)$. Applying Markov's inequality (see, e.g.,~\cite{mitzenmacher2017}) we have that
\[\mathbb{P}\left(S\sqsubset H\right)=\mathbb{P}\left(Y\geq1\right)\leq\mathbb{E}(Y)\leq\binom{n-r}{s-r}^{e(S)}p^{e(S)}\leq(n^{s-r}p)^{e(S)}.\]
To bound $\mathbb{P}\left(|X_{\overline{C(S)}}|\leq\delta n^sp/2\right)$, we recall that there are at least $\delta n^s$ copies of $K^{(r)}_s$ in $\overline{C(S)}$. Thus, $\mathbb{E}(|X_{\overline{C(S)}}|)\geq\delta n^sp$. Then, by Chernoff's bound~\cite{mitzenmacher2017},
\[\mathbb{P}\left(|X_{\overline{C(S)}}|\leq\delta n^sp/2\right)\leq\mathbb{P}\left(|X_{\overline{C(S)}}|\leq(1-1/2)\mathbb{E}(|X_{\overline{C(S)}}|)\right)\leq\exp(-\varepsilon n^sp),\]
where $\varepsilon=\delta/100$. Plugging the above upper bounds for $\mathbb{P}\left(S\sqsubset H\right)$ and $\mathbb{P}\left(|X_{\overline{C(S)}}|\leq\delta e(H)\right)$ back in~\eqref{P_bad}, we obtain that $\mathbb{P}_\text{bad}\leq o(1)+\exp(-\varepsilon n^sp)\cdot\sum_{S\in\mathcal{S}}(n^{s-r}p)^{e(S)}$. Let $\beta=\alpha n^{r-1/m_r(F)}$ and recall that $e(S)\leq\beta$ for all $S\in\mathcal{S}$. Then
\begin{align*}
\sum_{S\in\mathcal{S}}(n^{s-r}p)^{e(S)}\leq\sum_{1\leq k\leq\beta}\binom{n^r}{k}(n^{s-r}p)^k\leq\beta\left(\frac{3n^r}{\beta}\right)^\beta(n^{s-r}p)^\beta=\exp\left(\log{\beta}+\beta\log{\frac{3n^sp}{\beta}}\right),
\end{align*}
where in the second inequality we used $\binom{n^r}{k}(n^{s-r}p)^k\leq\binom{n^r}{\beta}(n^{s-r}p)^\beta$ for $k\leq\beta$ and $\binom{n^r}{\beta}\leq(3n^r/\beta)^\beta$. Since $n^sp\gg\beta$, we have that $\log{\beta}+\beta\log{\frac{3n^sp}{\beta}}\ll n^sp$ and thus $\sum_{S\in\mathcal{S}}(n^{s-r}p)^{e(S)}\leq\exp\left(o(n^sp)\right)$. Therefore,
\[\mathbb{P}_\text{bad}\leq o(1)+\exp(-\varepsilon n^sp)\cdot\sum_{S\in\mathcal{S}}(n^{s-r}p)^{e(S)}\leq o(1)+\exp\left(-\varepsilon n^sp+o(n^sp)\right),\]
which tends to $0$ as $n$ tends to infinity. This completes the proof.\qed

\section{Proof of~\Cref{random}: the second statement}
\label{second_statement}
In this section we prove the statement~\eqref{(ii)} of~\Cref{random}. We first establish two technical lemmas.

Let $r\geq2$ be an integer and let $W$ be a set of at least $r$ elements. A collection $\mathcal{E}$ of sets is called an $r$-cover of $W$ if, for every $r$-subset $B\subseteq W$, there exists an $A\in\mathcal{E}$ such that $B\subseteq A$. We say that an $r$-cover $\mathcal{E}$ of $W$ is trivial if $|\mathcal{E}|=1$.

\begin{lemma}
\label{bound_r_cover}
Let $t>r\geq2$ be integers. For any set $W$ of $t$ elements, if $\mathcal{E}\subseteq 2^W$ is a minimal non-trivial $r$-cover of $W$, then
\[(|\mathcal{E}|-1)\left(r-\frac{1}{m_r(K^{(r)}_t)}\right)-\sum_{A\in\mathcal{E}}|A|\leq-t.\]
\end{lemma}
\begin{proof}[Proof of~\Cref{bound_r_cover}]
Let $W$ be a set of $t$ elements and let $\mathcal{E}=\{A_1,\dots,A_{|\mathcal{E}|}\}\subseteq2^W$ be a minimal non-trivial $r$-cover of $W$. Let $\varphi(\mathcal{E})=(|\mathcal{E}|-1)/m_r(K^{(r)}_t)+\sum_{A\in\mathcal{E}}(|A|-r)$. The statement follows if we show that $\varphi(\mathcal{E})\geq t-r$.

Due to the minimality of $\mathcal{E}$, $|A_i|\geq r$ holds for all $1\leq i\leq|\mathcal{E}|$. We shall construct a sequence $\mathcal{E}_0,\mathcal{E}_1,\dots,\mathcal{E}_{|\mathcal{E}|}$ of $r$-covers of $W$ as follows. Let $\mathcal{E}_0=\mathcal{E}$, and then for each $1\leq i\leq|\mathcal{E}|$, let \[\mathcal{E}_i=\left(\mathcal{E}_{i-1}\setminus\{A_i\}\right)\cup\binom{A_i}{r}.\]
Observe that $\mathcal{E}_{|\mathcal{E}|}=\binom{W}{r}$, and recall that $m_r(K^{(r)}_t)=\frac{\binom{t}{r}-1}{t-r}$. Thus,
\[\varphi(\mathcal{E}_{|\mathcal{E}|})=\frac{\binom{|W|}{r}-1}{m_r(K^{(r)}_t)}=t-r.\]
It remains to show that $\varphi(\mathcal{E}_i)\leq\varphi(\mathcal{E}_{i-1})$ holds for any $1\leq i\leq|\mathcal{E}|$. Since $|\mathcal{E}_i|\leq|\mathcal{E}_{i-1}|-1+\binom{|A_i|}{r}$ by construction, we have that
\begin{align*}
\varphi(\mathcal{E}_i)=\frac{|\mathcal{E}_i|-1}{m_r(K^{(r)}_t)}+\sum_{A\in\mathcal{E}_i}(|A|-r)&\leq\frac{|\mathcal{E}_{i-1}|-1}{m_r(K^{(r)}_t)}+\sum_{A\in\mathcal{E}_{i-1}}(|A|-r)+\left(\frac{\binom{|A_i|}{r}-1}{m_r(K^{(r)}_t)}-(|A_i|-r)\right)\\
&=\varphi(\mathcal{E}_{i-1})+\left(\frac{\binom{|A_i|}{r}-1}{m_r(K^{(r)}_t)}-(|A_i|-r)\right).
\end{align*}
We conclude the proof by showing that $\frac{\binom{|A_i|}{r}-1}{m_r(K^{(r)}_t)}-(|A_i|-r)\leq 0$. If $|A_i|=r$, this is immediate. If $|A_i|>r$, then since $m_r(K^{(r)}_t)\geq m_r(K^{(r)}_{t'})=\frac{\binom{t'}{r}-1}{t'-r}$ holds for all $t\geq t'>r$, it follows that $\frac{\binom{|A_i|}{r}-1}{m_r(K^{(r)}_t)}\leq |A_i|-r$. This completes the proof.
\end{proof}

\begin{lemma}
\label{bad_cover}
Let $s\geq t>r\geq2$ be fixed integers and let $H=H(n,s,p)$ with $p\ll n^{r-s-1/m_r(K^{(r)}_t)}$. For any $W\in\binom{V(H)}{t}$, denote by $X_W$ the number of minimal non-trivial $r$-covers $\mathcal{E}\subseteq2^{E(H)}$ of $W$. Then $\mathbb{E}(X_W)\ll pn^{s-t}$.
\end{lemma}
\begin{proof}[Proof of~\Cref{bad_cover}]
Fix an arbitrary $W\in\binom{V(H)}{t}$. Let $\mathcal{W}$ be the collection of all minimal non-trivial $r$-covers $\mathcal{E}\subseteq2^{W}$ of $W$. Let $\mathcal{X}$ be the collection of all minimal non-trivial $r$-covers $\mathcal{E}\in2^{E(H)}$ of $W$. Observe that for every $\mathcal{E}=\{A_1,\dots,A_k\}\in\mathcal{X}$, $\{A_1\cap W,\dots,A_k\cap W\}\in\mathcal{W}$. Then,
\begin{align*}
\mathbb{E}(X_W)=\sum_{\mathcal{E}\in\mathcal{X}}p^{|\mathcal{E}|}\leq\sum_{\mathcal{E}\in\mathcal{W}}\prod_{A\in\mathcal{E}}\binom{n}{s-|A|}p^{|\mathcal{E}|}&\leq\sum_{\mathcal{E}\in\mathcal{W}}n^{|\mathcal{E}|s-\sum_{A\in\mathcal{E}}|A|}p^{|\mathcal{E}|}\\
&=pn^s\left(\sum_{\mathcal{E}\in\mathcal{W}}n^{(|\mathcal{E}|-1)s-\sum_{A\in\mathcal{E}}|A|}p^{|\mathcal{E}|-1}\right).
\end{align*}

Since $p\ll n^{r-s-1/m_r(K^{(r)}_t})$ and, by~\Cref{bound_r_cover}, $(|\mathcal{E}|-1)(r-1/m_r(K^{(r)}_t))-\sum_{A\in\mathcal{E}}|A|\leq-t$ holds for all $\mathcal{E}\in\mathcal{W}$, it follows that
\[\sum_{\mathcal{E}\in\mathcal{W}}n^{(|\mathcal{E}|-1)s-\sum_{A\in\mathcal{E}}|A|}p^{|\mathcal{E}|-1}\ll\sum_{\mathcal{E}\in\mathcal{W}}n^{(|\mathcal{E}|-1)(r-1/m_r(K^{(r)}_t))-\sum_{A\in\mathcal{E}}|A|}=|\mathcal{W}|n^{-t}.\]
As $|\mathcal{W}|=O(1)$, we have that $\mathbb{E}(X_W)\ll pn^{s-t}$.
\end{proof}

\subsection{Proof of~\Cref{random} statement~\eqref{(ii)}}
Let $s\geq t>r\geq2$ be fixed integers. Let $H=H(n,s,p)$ with $n^{-s}\ll p\ll n^{r-s-1/m_r(K^{(r)}_t)}$.
Observe that if an $s$-subgraph $H_0\subseteq H$ is not $(r,t)$-conformal, then there must exist a minimal non-trivial $r$-cover $\mathcal{E}\in2^{E(H_0)}$ of some $W\in\binom{V(H_0)}{t}$. Recall that for each $W\in\binom{V(H)}{t}$, we let $X_W$ denote the number of minimal non-trivial $r$-covers $\mathcal{E}\subseteq2^{E(H)}$ of $W$. Let $X$ be the sum of $X_W$ over all $W\in\binom{V(H)}{t}$. Moreover, let $Y$ be the number of pairs of edges $\{A_1,A_2\}\in\binom{E(H)}{2}$ satisfying $|A_1\cap A_2|\geq r$.

We shall prove that $X+Y\ll e(H)$ holds with high probability. It then follows that, with high probability, there exists an $r$-linear $(r,t)$-conformal $H_0\subseteq H$ with $e(H_0)=(1-o(1))e(H)$, obtained by deleting one edge from each configuration counted by $X$ and $Y$.

By~\Cref{bad_cover} we have that
\[\mathbb{E}(X)=\mathbb{E}\left(\sum_{W\in\binom{V(H)}{t}}X_W\right)=\sum_{W\in\binom{V(H)}{t}}\mathbb{E}(X_W)\ll\binom{n}{t}pn^{s-t}\leq pn^s.\]
Moreover, since $p\ll n^{r-s-1/m_r(K^{(r)}_t)}$,
\[\mathbb{E}(Y)\leq\binom{n}{s}\binom{s}{r}\binom{n}{s-r}p^2\leq n^{2s-r}p^2=pn^s(n^{s-r}p)\ll pn^s.\]
Let $f=f(n)$ be such that $\mathbb{E}(X)+\mathbb{E}(Y)\ll f\ll pn^s$. Then by Markov's inequality~\cite{mitzenmacher2017}
\[\mathbb{P}(X+Y\geq f)\leq\frac{\mathbb{E}(X+Y)}{f}=\frac{\mathbb{E}(X)+\mathbb{E}(Y)}{f}=o(1).\]
Furthermore, $\mathbb{E}(e(H))=\binom{n}{s}p\gg1$ since $p\gg n^{-s}$. Then by Chernoff's inequality~\cite{mitzenmacher2017},
\[e(H)=\Omega\left(\mathbb{E}(e(H))\right)=\Omega(pn^s)\]
holds with high probability. Therefore, we have that with high probability $X+Y\ll e(H)$.\qed

\section{Concluding remarks}
\label{concluding}
In this note, we extend recent work of Mendon\c{c}a, Miralaei, and Mota~\cite{mendonca2025} from the graph setting to uniform hypergraphs. To construct an $r$-graph $G$ with asymmetric Ramsey properties involving complete $r$-graphs, we first consider a random hypergraph $H$ of higher uniformity and then show that, with high probability, $H$ contains a dense subhypergraph whose primal $r$-graph exhibits the desired properties. The argument combines probabilistic and hypergraph container methods.

Several natural questions remain open. First, it would be interesting to study whether analogous statements hold for more general classes of hypergraphs beyond complete $r$-graphs. More precisely, for any $r$-graphs $T_1,T_2,F$ with $m_r(T_1)\geq m_r(T_2)>m_r(F)$, does there exist an $r$-graph $G$ such that
\[G\not\to(T_1,T_2)\quad\text{but}\quad G\to(K^{(r)}_{R(T_1,T_2)-1},F)?\]
The primal $r$-graph approach may no longer be suitable in this setting, since the primal $r$-graph is always a union of large cliques. Second, as our approach is probabilistic, it would be interesting to obtain an explicit construction.

\subsection*{Acknowledgement}
The author is grateful to Dingyuan Liu for his valuable advice and insightful discussions.


\begin{thebibliography}{99}

\bibitem{balogh2015}
J.~Balogh, R.~Morris, and W.~Samotij,
Independent sets in hypergraphs,
\emph{Journal of the American Mathematical Society} \textbf{28} (2015), 669--709.

\bibitem{christoph2025}
M.~Christoph, A.~Martinsson, R.~Steiner, and Y.~Wigderson,
Resolution of the Kohayakawa--Kreuter conjecture,
\emph{Proceedings of the London Mathematical Society} \textbf{130} (2025).

\bibitem{gugelmann2017}
L.~Gugelmann, R.~Nenadov, Y.~Person, N.~\v{S}kori\'c, A.~Steger, and H.~Thomas,
Symmetric and asymmetric Ramsey properties in random hypergraphs,
\emph{Forum of Mathematics, Sigma} \textbf{5} (2017).

\bibitem{mendonca2025}
W.~Mendon\c{c}a, M.~Miralaei, and G.~O.~Mota,
Graphs with asymmetric Ramsey properties,
arXiv:2511.02963 (2025).

\bibitem{mitzenmacher2017}
M.~Mitzenmacher and E.~Upfal,
Probability and Computing: Randomized Algorithms and Probabilistic Analysis,
Cambridge University Press, Cambridge, 2017.

\bibitem{morris2026}
R.~Morris,
Some recent results in Ramsey theory,
arXiv:2601.05221 (2026).

\bibitem{mubayi2020}
D.~Mubayi and A.~Suk,
A survey of hypergraph Ramsey problems,
\emph{Discrete Mathematics and Applications} (2020), 405--428.

\bibitem{nenadov2016}
R.~Nenadov and A.~Steger,
A short proof of the random Ramsey theorem,
\emph{Combinatorics, Probability and Computing} \textbf{25} (2016), 130--144.

\bibitem{rodl1995}
V.~R\"odl and A.~Ruci\'nski,
Threshold functions for Ramsey properties,
\emph{Journal of the American Mathematical Society} \textbf{8} (1995), 917--942.

\bibitem{saxton2015}
D.~Saxton and A.~Thomason,
Hypergraph containers,
\emph{Inventiones mathematicae} \textbf{201} (2015), 925--992.

\end{thebibliography}
\end{document}